\documentclass[a4paper,reqno,11pt]{amsart} 
	\usepackage[latin1]{inputenc}
    \usepackage[T1]{fontenc}
    \usepackage[english]{babel}
    \usepackage{appendix}
    \usepackage[english]{minitoc}
    \usepackage[nice]{nicefrac}

\numberwithin{equation}{section}
\makeatletter\@addtoreset{equation}{section}
	\usepackage{amsmath}
    \usepackage{amssymb}
    \usepackage{amsbsy} 
    \usepackage{latexsym, amsfonts}
    \usepackage{graphics}
    \usepackage{ulem}
    \usepackage{hhline}
    \usepackage{dsfont}
    \usepackage{mathrsfs}
    \usepackage{color}
    \usepackage{fancyhdr}
    \usepackage{rotating}
    \usepackage{fancybox}
    \usepackage{colortbl}
    \usepackage{pifont}
    \usepackage{setspace}
    \usepackage{enumerate}
    \usepackage{multicol}
    \usepackage{varioref}
    \usepackage{lmodern}
    \usepackage{textcomp}
    \usepackage{euscript}
    \usepackage[pdftex]{hyperref}
    
			\newtheorem{theorem}{Theorem}[section]
			\newtheorem{definition}[theorem]{Definition}
			
			\newtheorem{lemma}[theorem]{Lemma}
			\newtheorem{proposition}[theorem]{Proposition}
			
			\newtheorem{corollary}[theorem]{Corollary}
			\newtheorem{remark}[theorem]{Remark}

\usepackage{xcolor}

\newcommand{\C}{\mathbb C}    \newcommand{\R}{\mathbb R}    	 
\newcommand{\Hq}{\mathbb H}  \newcommand{\Sq}{\mathbb S}
\newcommand{\bq}{\overline{q} }

    \newcommand{\bz}{\overline{z}}

\begin{document}
	\title[Polyregularity of the dot product of slice regular functions]{Polyregularity of the dot product of slice regular functions}

\thanks{Dedicated to the memory of Kettani Ghanmi (1945--2018)}

\author{A. Ghanmi }
\address{A.G.S.-L.A.M.A., CeReMAR, 
          Department of Mathematics,\newline
          P.O. Box 1014,  Faculty of Sciences, 
          Mohammed V University in Rabat, Morocco}
 \email{allal.ghanmi@um5.ac.ma / ag@fsr.ac.ma}


\date{January 18, 2017}

\begin{abstract}
	In this paper, we are concerned with the S--polyregularity the regular dot product of slice regular functions. We prove that the product of a slice regular function and right quaternionic polynomial function is a S--polyregular function and we determinate its exact order. The general case of the product of any two slice regular functions is also discussed. In fact, we provide sufficient and necessary conditions to the product of slice regular functions be a S--polyregular function.
	The obtained results are then extended to the product of S--polyregular functions and remain valid for a special dot product. As consequences we obtain linearization theorems for such S--polyregular products with respect to the slice regular functions.
\end{abstract}


\maketitle

\section{Introduction}

A special generalization of entire functions are the so--called $(n+1)$--polyanalytic functions. We denote their space by $\mathcal{A}_n$.
According to \cite{Burgatti1922}, a function $f$ is said to be polyanalytic of order $n+1$ (level $n$) if it satisfies the generalized Cauchy equation $\partial^{n+1}_{\bz} f=0$. 
Such functions have numerous applications in various fields and the sphere of applications has been considerably extended to mechanics and mathematics physics. See for instance \cite{Kolossov1908,Muskhelishvili1968,
	Wendland1979
	,HaimiHedenmalm2013,AbreuFeichtinger2014,AbreuBalazsGossonMouayn2015}.
For a complete survey on these functions one can refer to \cite{Balk1991} and the rich list of references therein.
In the last decade, they have found interesting applications in signal processing, time--frequency analysis and
wavelet theory \cite{Grochenig2001,GrochenigLyubarskii2009,AbreuFeichtinger2014}. Moreover, polyanaliticity was essential in characterizing the classical Hilbert space on the complex plane with respect to the gaussian measure in terms of the so--called polyanalytic Bargmann spaces \cite{Vasilevski2000}. Such spaces are closely connected to the spectral theory of a special magnetic Laplacian \cite{GI-JMP2005,AbreuBalazsGossonMouayn2015,AbreuFeichtinger2014}.

The first generalization in the quaternionic setting is due to Brackx \cite{Brackx1976a,Brackx1976b} who defined  the $k$--monogenic functions to be the elements of $Ker(D)^{k+1}$, $D$ being the F\"uter operator. Its analog for the (left) slice derivative $\overline{\partial_I}$ is recently introduced in \cite{ElHamyani2017} (see also \cite{ElHamyani2018,BenElHam2018,AlpayDS2019}). 
Thus, the solution of the generalized Cauchy--Riemann equations $\overline{\partial_I}^{n+1} f|_I=0$; $I^2=-1$, leads to the space $\mathcal{SR}_n$ of S--polyregular functions of level $n$ (or order $n+1$), i.e., such that the restriction to any slice $\C_I$ is polyanalytic of level $n$.
Some basic properties of this new class of S--polyregular functions 
have been established and investigated in the recent works \cite{BenElHam2018,AlpayDS2019}.
For our purpose, we recall the following 

\begin{proposition}[Proposition 2.3 in \cite{BenElHam2018}]\label{propSPDecom}
	For every $f\in \mathcal{SR}_n$ there exist some  $\varphi_k \in \mathcal{SR}= \mathcal{SR}_0$; $k=0,1,\cdots, n$, such that
	$$f(q,\overline{q})=\sum_{k=0}^{n}\overline{q}^k\varphi_k(q).$$
\end{proposition}

\begin{proposition}[Proposition 2.4 in \cite{BenElHam2018}]\label{HermiteOp}
	A function $f$ belongs to $\mathcal{SR}_n$ if and only if there exists some $F\in \mathcal{SR}$ such that $f=\mathcal{H}_{n} (F)$. Moreover, we have
	$$  \mathcal{SR}_n=	\sum_{k=0}^n\mathcal{H}_{k} (\mathcal{SR}),$$
	 where
	the differential  transformation $\mathcal{H}_{n}$ is defined by
	$$\mathcal{H}_{k} (F)](q) :=  (-1)^k e^{|q|^2} \partial^k_I(e^{-|q|^2}F).$$
\end{proposition}

\begin{proposition}[Proposition 2.5 in \cite{BenElHam2018}]
	\label{SplitSp}
	If $f$ is a S--polyregular function, then for every
	$I\in \mathbb{S}$, and every $J \in \mathbb{S}$ perpendicular to $I$, there are two polyanalytic functions $F_f$, $G_f$ : $ \mathbb{C}_I \longrightarrow \mathbb{C}_I$ such that for any $q = x + I y$
	$$ f_I (q) = F_f(q) + G_f(q)J .$$
\end{proposition}


Although the slice regularity is a beautiful notion allowing  generalization of interesting results from complex analysis to quaternionic analysis, there are many technical problems to take care of for lack of commutativity. A known fact in this context is that the product of two left slice regular functions is not necessarily a slice regular.
To overcome this, a special $\star^L_s$--product was introduced in \cite{GentiliStoppato08}.
Its natural extension to S--polyregular functions is presented in \cite{BenElHam2018}.
However, the most common classical examples, always used to prove that the dot product do not leave $\mathcal{SR}$ invariant, are clearly S--polyregulars as is the case of the elementary function $f(q,\bq)=(q-\mathbf{i})q$. Indeed, $\overline{\partial_I}^{2}(f|_I) =0$. Moreover, we have
$$ f|_I = \bq \varphi|_I + \psi|_I; \qquad \varphi|_I:= -\frac{1}{2} (I+I \mathbf{i}  I) .$$
The function $\psi$ defined by $\psi|_I = f|_I - \bq \varphi|_I$ belongs to $\mathcal{SR}$ for $\overline{\partial_I}(\mathbf{i} q)=\varphi$.

Motivated by this observation, we will study in this paper the polyregularity of slice regular functions leading to a special phenomena for the
classical dot product in $\mathcal{SR}$.
More precisely, we are interested in providing necessary and sufficient conditions on the input functions $f,g\in \mathcal{SR}$ to $f\cdot g$ be S--polyregular.
Thus, we begin by proving that the generic element $q^n\alpha q^m$; $\alpha\in \Hq$, is S--polyregular of level do not exceed $m$. Its exact level is specified in Theorem \ref{thmASreg}. More generally, we show that if $f\in \mathcal{SR}$ and  $g\in \Hq_M^R[q]$ is a right quaternionic polynomial function of degree $M$, then $f\cdot g \in \mathcal{SR}_M$ and the exact value of its level has been made explicit (see Theorem \ref{thmpolynomial}). This result is improved in Section 3 by showing that the sufficient condition $g\in \Hq_M^R[q]$ is also a necessary condition (Theorem \ref{thm:char2}).
Another characterization is given by Theorem \ref{thmProdSpR} which proves that the dot product $f\cdot g$ of $f,g\in \mathcal{SR}$ is a S--polyregular function if and only if $f,g\in \mathcal{SR}$ satisfy one of the following
equivalent conditions
$$ (a) \quad \mathcal{Z}(G_f) \cup \mathcal{Z}\left( \dfrac{\partial^{k_0}g_{_{I}} }{\partial y^{k_0}}\right) = \C_I \Longleftrightarrow
 (b) \quad G_f J \dfrac{\partial^{k_0} g_{_{I}} }{\partial y^{k_0}}=0 \mbox{ on } \C_I$$
for some positive integer $k_0$ and every $I\in\Sq$.
Here
$\mathcal{Z}(h)$ denotes the zeros set of $h$ and $f_{_{I}} =F_f + G_f J$ with
$F_f,G_f:\C_{I}\longrightarrow \C_{I}$ are two holomorphic functions on $\C_{I}$; $I\in \Sq$ and $J$ is imaginary unit perpendicular to $I$.
Subsequently, a sufficient condition to $f\cdot g $ be slice regular is shown to be  $f_{_{I}} \in\C_I$ for every $I\in\Sq$. Another immediate consequence is that the resulting products can be linearized with respect to $\mathcal{SR}$ in the sense that for every $f,g\in \mathcal{SR}$ such that $f\cdot g\in \mathcal{SR}_m$,
 we have $ f \cdot g = \sum\limits_{j=0}^m\bq^j \varphi_j$ for some $\varphi_j \in \mathcal{SR}$.

The generalization to S--polyregular functions is possible by considering  the special $\odot ^L_{sp}$--product defined by
\begin{equation}
(f\odot ^L_{sp} g) (q,\bq) = \sum\limits_{j,k=0}^{m,n} \bq^{j+k} \varphi_j (q)\psi_k(q)
\end{equation}
for given $f(q,\overline{q})=\sum\limits_{j=0}^{m}\overline{q}^j\varphi_j \in \mathcal{SR}_m$ and 
$ g(q,\overline{q})=\sum\limits_{k=0}^{n}\overline{q}^k\psi_k \in \mathcal{SR}_n$ with
$\varphi_j,\psi_k \in \mathcal{SR}$.
This will be discussed in Section 4.

\section{{S--polyregularity of $\mathcal{SR} \cdot \Hq_M^R[q]$}}

We begin by examining the S--polygularity of the dot product of elementary slice regular functions presented by the generic elements
\begin{equation}\label{DefA}
A^{n,m}(q|\alpha) := q^n \alpha q^m ; \quad q\in \Hq,
\end{equation}
for given nonnegative integers $n,m$ and quaternionic number $\alpha$.
Thus, making direct computation we see that
the $A^{n,m}(q|\alpha)$ satisfy the identities
\begin{align}
\overline{\partial_I} (A^{n,m}(q|\alpha))& = \frac m 2 \left(A^{n,m-1}(q|\alpha) - I A^{n,m-1}(q|\alpha I) \right) \nonumber\\
& = \frac m 2 \left(A^{n,m-1}(q|\alpha) +  A^{n,m-1}(q|I\alpha I) \right) \nonumber\\ 
& = \frac m 2  A^{n,m-1}(q|\alpha + I\alpha I) . \label{idA3}
\end{align}

Now, associated to given nonnegative integers $m$ and $k$, we consider the function on $\Sq\times \Hq $ given by
\begin{equation}\label{DefS}
S^m_k(I,\alpha) := \sum\limits_{j=0}^k\binom{k}{j} I^j \alpha I^j .
\end{equation}
Thus, we claim the following

\begin{lemma}\label{lemS} Let $I\in \Sq$ and $m,k$ be nonegative integers. Then
	\begin{itemize}
		\item[(i)] We have $S^m_k(I,\alpha)=0 $ for every real $\alpha$.
		\item[(ii)] The functional $ \alpha \longmapsto S^m_k(I,\alpha)$ is $\R$--linear on $\Hq$, in the sense that
		\begin{align}
		S^m_k(I,\alpha+ \lambda \beta ) = S^m_k(I,\alpha) + S^m_k(I,\beta ) \lambda  \label{idS2}
		\end{align}
		holds true for every $\alpha,\beta\in \Hq$ and $\lambda\in \R$.
		\item[(iii)] We have the binomial property
		\begin{align}
		S^m_k(I,\alpha) + I S^m_k(I,\alpha) I = S^m_{k+1}(I,\alpha) .  \label{idS1}
		\end{align}
		\item[(iv)] If $S^m_{k_0}(I,\alpha)=0 $ for some nonnegative integer ${k_0}$, then $S^m_k(I,\alpha)=0 $ for every integer $k\geq {k_0}$.
	\end{itemize}
\end{lemma}

Such function will play a crucial rule in determining the exact level of S--polyregularity of $A^{n,m}(q|\alpha)$. To be more precise, we need to fix further objects.
Let $\aleph^I_m(\alpha)$ denotes the subset of nonnegative integers
$$ \aleph^I_m(\alpha) : = \{ k; k=0,1, \cdots,  S^m_k(I,\alpha) =0\}.$$
From $(iv)$ of Lemma \ref{lemS}, it is clear that $\aleph^I_m(\alpha)$ is infinite once it is not empty.
To $\aleph^I_m(\alpha)$, we associate the quantity
$$\varrho_{sp}^{m}(\alpha) := \max\limits_{I\in \Sq}\{\rho_{sp}^{m,I}(\alpha)  \},$$
where $\rho_{sp}^{m,I}(\alpha)$ stands for
$$
\rho_{sp}^{m,I}(\alpha) := \left\{ \begin{array}{ll} \min\{m, \min \aleph^I_m(\alpha) \} & \mbox{ if } \quad \aleph^I_m(\alpha)\ne \emptyset,\\
m & \mbox{ if } \quad \aleph^I_m(\alpha)= \emptyset.
\end{array}
\right.
$$
Notice for instance that $\varrho_{sp}^{m}(\alpha)=0$ for $\rho_{sp}^{m,I}(\alpha)=0$ whenever $\alpha$ is real.

\begin{theorem}\label{thmASreg}
	The function $A^{n,m}(q|\alpha)$ is S--polyregular in the $q$--variable of exact level
	$$Lev_{sp}(A^{n,m}(\cdot|\alpha)) =\varrho_{sp}^{m}(\alpha) \leq m.$$
\end{theorem}

\begin{proof}
	To prove the S--polyregularity of $A^{n,m}(q|\alpha)$
	we make use of \eqref{idA3} and a reasoning by induction. Indeed, we show that
	\begin{align}\label{partialAk}
	\overline{\partial_I}^k (A^{n,m}(q|\alpha)) &= \left(\frac{1}{2}\right)^k \frac{m!}{(m-k)!} A^{n,m-k}(q|	S^m_k(I,\alpha))\\& = \left(\frac{1}{2}\right)^k \frac{m!}{(m-k)!} q^n	S^m_k(I,\alpha) q^{m-k} \nonumber
	\end{align}
	for every $k\leq m$, where $S^m_k(I,\alpha)$ is the function in \eqref{DefS}.
	Therefore, for $k=m$ we have
	$$\overline{\partial_I}^m (A^{n,m}(q|\alpha)) = \left(\frac{1}{2}\right)^m m! A^{n,0}(q|	S^m_m(I,\alpha))\in \mathcal{SR},$$
	so that $\overline{\partial_I}^{m+1} (A^{n,m}(q|\alpha)) =0$.
	Furthermore, for every $I\in \Sq$ we have 
	$$ \overline{\partial_I}^{\rho_{sp}^{m,I}(\alpha)} (A^{n,m}(q|\alpha)) \ne 0 \quad \mbox{and} \quad \overline{\partial_I}^{\rho_{sp}^{m,I}(\alpha)+1} (A^{n,m}(q|\alpha)) = 0$$  by the definition of $\rho_{sp}^{m,I}(\alpha)$.
		This is to say
		$A^{n,m}(q|\alpha)  \in \mathcal{SR}_{\varrho_{sp}^{m}(\alpha)} \subset  \mathcal{SR}_{m}$ of exact level $\varrho_{sp}^{m}(\alpha)$.
	This completes the proof.
\end{proof}

\begin{corollary}[Linearization]\label{corLin}
	The generic element $A^{n,m}(q|\alpha)$ is a polynomial function in $\bq$ with coefficients in $\mathcal{SR}$, that is there exist some slice regular functions $\varphi^\alpha_j$; $j=0, 1,\cdots , m$, such that
	$$ A^{n,m}(q|\alpha) = \bq^m \varphi^\alpha_m + \bq^{m-1}\varphi^\alpha_{m-1} + \cdots + \bq \varphi^\alpha_1+\varphi^\alpha_0.$$
\end{corollary}

\begin{proof} The proof is an immediate consequence of the S--polyregularity of $A^{n,m}(q|\alpha)$ proved in Theorem \ref{thmASreg} combined with Proposition 1.1. 
\end{proof}

\begin{definition}
	We call $\varrho_{sp}^{m}(\alpha)$ the S--polyregularity index of the generic element $A^{n,m}(q|\alpha)$. 	
\end{definition}

The $\varrho_{sp}^{m}(\alpha)$--index is in general hard to compute. However, we can provide exact value of $\varrho_{sp}^{m}(\alpha)$ for some particular  cases. Thus, for real $a$, we have $\aleph^I_m( a) =\{0\}$ and then $\varrho_{sp}^{m}( a)=0$. The case of $\alpha= a+bJ$; $J\in \Sq$, $a\in\R$ and $b\in\R^*$, is discussed by the following lemma.

\begin{lemma}
	For every $J\in \Sq$, $a\in\R$ and $b\in\R^*$, we have $\aleph^I_m(a+bJ) = \aleph^I_m(J)$ and
	\begin{equation}
	\varrho_{sp}^{m}(a+bJ) = \varrho_{sp}^{m}(J)= m.
	\end{equation}
\end{lemma}

\begin{proof} Using the $\R$--linearity of the functional $ \alpha \longmapsto S^m_k(I,\alpha)$, we get
	\begin{align*}
	\aleph^I_m(a+bJ)
	&	= \{ k; k=0,1, \cdots,  S^m_k(I,a) + S^m_k(I,J) b =0\}
	=\aleph^I_m(J).
	\end{align*}
	Now, if $K\in \Sq$ such that $K\perp J$, then $\aleph^K_m(J)=\emptyset$ since in this case $K^jJ=(-1)^jJK^j$ and therefore
	$$S^m_k(K,J) = \sum\limits_{j=0}^k\binom{k}{j} (-1)^j I K^{2j} = \sum\limits_{j=0}^k\binom{k}{j} I \ne 0$$ for every $k$.
	Accordingly,  $\rho_{sp}^{m,K}(J) = m$. To conclude, it suffices to observe that
	$\rho_{sp}^{m,I}(J) \leq \varrho_{sp}^{m}(J) \leq m$ holds true for every $I\in \Sq$ and in particular for $I=K\perp J$.
\end{proof}

The next result discuss the generalization of Theorem \ref{thmASreg}
to the dot product of a slice regular function $f\in \mathcal{SR}$ and a right quaternionic polynomial function $g\in \Hq_M^R[q]$. The obtained result can be stated as follows.

\begin{theorem} \label{thmpolynomial}
	Let $f(q) = \sum\limits_{n=0}^\infty q^n \alpha_n$ be a slice regular function and $g$ a right quaternionic polynomial function of degree $M$, then $f\cdot g $ is a S--polyregular function whose exact level given by
	$$ Lev_{sp}(f\cdot g) = \min \{M, \max\limits_{I\in\Sq} \min \bigcap_{m=k}^{M} \bigcap_{n=0}^\infty \aleph^I_{m}(\alpha_n \}.$$	
\end{theorem}

\begin{proof}
	Notice for instance that if $ f(q) = \sum\limits_{n=0}^N q^n \alpha_n $ and $ 
	g(q) = \sum\limits_{m=0}^M q^m \beta_m,$ are right quaternionic polynomial functions of degrees $N$ and $M$, respectively, then we have
	$$(f\cdot g)(q)
	=  \sum\limits_{n=0}^N\sum\limits_{m=0}^M A^{n,m}(q|\alpha_n) \beta_m.$$
 Hence, by Theorem \ref{thmASreg}, it is clear that $f\cdot g$ is  S--polyregular of level do not exceed $M$. Its exact level satisfies
	$$Lev_{sp}(f\cdot g) \leq \max_{n=0,\cdots, N,\atop m=0,\cdots,M}
	\{  \varrho_{sp}^{m}(\alpha_n)\} 
	\leq M .$$
The case of arbitrary $ f(q) = \sum\limits_{n=0}^N q^n \alpha_n \in \mathcal{SR}$ and $g$ is a right quaternionic polynomial function of degree $M$ is similar. In fact, from
	$$ (f\cdot g)(q)
		=    \sum\limits_{n=0}^\infty\sum\limits_{m=0}^M A^{n,m}(q|\alpha_n) \beta_m$$
	and Theorem \ref{thmASreg}, we claim that $f\cdot g \in \mathcal{SR}_M$ as well as
	$$Lev_{sp}(f\cdot g)
	\leq \max_{m=0,\cdots,M \atop n=0,\cdots  } Lev_{sp} (A^{n,m}(\cdot|\alpha_n))
	\leq \max_{m=0,\cdots,M \atop n=0,\cdots } \{  \varrho_{sp}^{m}(\alpha_n)\}.$$
This can be reproved rigourously making use of the identity \eqref{partialAk}. Indeed, direct computation yields
	\begin{align*} \overline{\partial_I}^k (f\cdot g)(q)
	&=\sum\limits_{n=0}^\infty\sum\limits_{m=0}^M \overline{\partial_I}^k(A^{n,m}(q|\alpha_n)) \beta_m
\\&	=\sum\limits_{n=0}^\infty\sum\limits_{m=k}^M \overline{\partial_I}^k(A^{n,m}(q|\alpha_n)) \beta_m
	\end{align*}
	since $\overline{\partial_I}^k(A^{n,m}(q|\alpha_n))=0$ for
	$m\leq k-1$. Next, using \eqref{partialAk} we obtain
	\begin{align} \label{fg} \overline{\partial_I}^k (f\cdot g)(q)
	&=
	\left(\frac{1}{2}\right)^k \sum\limits_{n=0}^\infty  q^n \sum\limits_{m=k}^M \frac{m!}{(m-k)!} 	S^m_k(I,\alpha_n) q^{m-k} \beta_m.
	\end{align}
	In particular, for $k=M$ it follows that
	\begin{align*} \overline{\partial_I}^M (f\cdot g)(q)
	&=\left(\frac{1}{2}\right)^M M! \sum\limits_{n=0}^\infty q^n S^{M}_M(I,\alpha)  \beta_{M} \in \mathcal{SR}.
	\end{align*}
Moreover, from \eqref{fg}, we see that the determination of the exact value of the S--polyregularity of $f\cdot g$ is closely connected to the set of positive integers $k$ such that
	$S^{m}_{k}(I,\alpha_n) =0$ for
	every $m=k,\cdots, M$ and $n=0,1,\cdots$, i.e.,
	$$ \aleph^I_{k,M}(f) := \bigcap_{m=k}^{M} \bigcap_{n=0}^\infty \aleph^I_{m}(\alpha_n).$$
	Indeed, if $\max\limits_{I\in\Sq} \min\limits \aleph^I_{M}(f) <M $, then $f\cdot g\in \mathcal{SR}_s$ with $s= \max\limits_{I\in\Sq} \min \aleph^I_{k,M}(f)$. Otherwise $s=M$.
	This completes our check of Theorem \ref{thmpolynomial}.
\end{proof}

\begin{remark}\label{}
	In virtue of $\mathcal{SR}_m = \sum\limits_{j=0}^m\bq^j \mathcal{SR} $, the previous theorem implies that $f \cdot g$ can be linearized with respect to $\mathcal{SR}$ in the sense that $f \cdot g =  \sum\limits_{j=0}^m\bq^j \varphi_j$ for some $\varphi_j \in \mathcal{SR}$. In particular, for every slice regular function $f$, the function $f\cdot q$ is biregular and therefore there exist $\varphi_0,\varphi_1 \in \mathcal{SR}$ such that $ f \cdot q = \bq \varphi_1 + \varphi_0$.
\end{remark}

\section{{A characterization of the S--polyregularity of $\mathcal{SR}\cdot \mathcal{SR}$}}

In the sequel, we will examin the general case of the dot product of slice regular functions in order to provide necessary and sufficient condition to such product  be  S--polyregular. The main result of the present section is based on the following lemma.

\begin{lemma}\label{lemActionk}
	Let  $f,g\in \mathcal{SR}$ and $F_f,G_f:\C_{I}\longrightarrow \C_{I}$ be two holomorphic functions on $\C_{I}$; $I\in \Sq$, such that
	$f_{_{I}} =F_f + G_f J$ for given imaginary unit $J$ perpendicular to $I$. Then, for every positive integer $k$, we have
	\begin{align}\label{lemprodkfgSplit}
	\overline{\partial_{I}}^k (fg)_{_{I}} &=  I^{k} G_f J  \dfrac{\partial^k g_{_{I}} }{\partial y^k} = G_f J  \dfrac{\partial^k g_{_{I}} }{\partial x^k}  .
	\end{align}
\end{lemma}

\begin{proof}
	A direct computation using the definition of the slice derivative shows that  $\overline{\partial_{I}} $ satisfies the  following rules for the product of real differentiable functions $f,g:\Hq \longrightarrow \Hq$,
\begin{equation}\label{Leibniz0}
\overline{\partial_{I}} (fg)_{_{I}} =(\overline{\partial_{I}} f_{_{I}}) g_{_{I}} +f_{_{I}}(\overline{\partial_{I}} g_{_{I}}) - \frac 12[f_{_{I}},I id_{\Hq}] \dfrac{\partial g_{_{I}} }{\partial y},
\end{equation}
where $[\cdot,\cdot]$ denotes the commutation operation. They reduce further to
\begin{equation*}\label{LeibnizSR}
\overline{\partial_{I}} (fg)_{_{I}} = - \frac 12[f_{_{I}},I id_{\Hq}] \dfrac{\partial g_{_{I}} }{\partial y}
\end{equation*}
if $f$ and $g$ are assumed to be slice regular functions.
By applying again \eqref{Leibniz0} to the product in the right--hand side of the obtained identity, we get
\begin{align*}
\overline{\partial_{I}}^2 (fg)_{_{I}} = - \frac 12
\left\{ (\overline{\partial_{I}}([f_{_{I}}, I id_{\Hq}] )) \dfrac{\partial g_{_{I}} }{\partial y} 
\right. & +  
[f_{_{I}}, I id_{\Hq}] (\overline{\partial_{I}} \dfrac{\partial g_{_{I}} }{\partial y}) 
\\& - \left. \frac 12 \left[ [f_{_{I}}, I id_{\Hq}] , I id_{\Hq} \right]   \dfrac{\partial^2 g_{_{I}} }{\partial y^2} \right\}
\end{align*}
which reduces further to
\begin{align*}
\overline{\partial_{I}}^2 (fg)_{_{I}} &=  \left(-\frac 12\right)^2
[ [f_{_{I}}, I id_{\Hq}] , I id_{\Hq}]   \dfrac{\partial^2 g_{_{I}} }{\partial y^2}
\\&= -\frac 12 (f_{_{I}} + If_{_{I}} I ) \dfrac{\partial^2 g_{_{I}} }{\partial y^2}
\end{align*}
by means of the  facts
$$ \overline{\partial_{I}} \left(\dfrac{\partial g_{_{I}} }{\partial y}\right) =
\dfrac{\partial }{\partial y} (\overline{\partial_{I}} g_{_{I}} )= 0$$
as well as
$$ \overline{\partial_{I}}([f_{_{I}}, I id_{\Hq}] ) = 
\overline{\partial_{I}}(f_{_{I}})I - \overline{\partial_{I}}(I f_{_{I}}) =
\overline{\partial_{I}}(f_{_{I}})I - I\overline{\partial_{I}}( f_{_{I}}) =0$$
fulfilled for $f,g\in \mathcal{SR}$. More generally, one gets by induction
\begin{equation*}
\overline{\partial_{I}}^k (fg)_{_{I}} =  \left(-\frac 12\right)^k
[f_{_{I}}, I id_{\Hq}]_k   \dfrac{\partial^k g_{_{I}} }{\partial y^k} ,
\end{equation*}
where $ [f_{_{I}}, I id_{\Hq}]_k:= [\cdots [ [f_{_{I}}, I id_{\Hq}] , I id_{\Hq}] \cdots , I id_{\Hq}]$ is the iteration of the commutator operator $k$--times. Straightforward  computation shows that
$$  [f_{_{I}}, I id_{\Hq}]_k = (-2I)^{k-1} [f_{_{I}}, I id_{\Hq}]; \qquad k\geq 1,$$ so that we have
\begin{equation}\label{proofprodkfg}
\overline{\partial_{I}}^k (fg)_{_{I}} =  -\frac 12 I^{k-1}
[f_{_{I}}, I id_{\Hq}]   \dfrac{\partial^k g_{_{I}} }{\partial y^k}
,
\end{equation}
By making use of the Splitting lemma \cite{
ColomboSabadiniStruppa2011,GentiliStoppatoStruppa2013,AlpayColomboSabadini2017}, there exist two holomorphic functions $F_f,G_f:\C_{I}\longrightarrow \C_{I}$ such that
$f_{_{I}} =F_f + G_f J$ for given imaginary unit $J$ perpendicular to $I$.
Therefore, we can rewrite \eqref{proofprodkfg} in the following form
\begin{equation*}
\overline{\partial_{I}}^k (fg)_{_{I}} =  I^{k} G_f J  \dfrac{\partial^k g_{_{I}} }{\partial y^k}
\end{equation*}
since  $F_f$ and $G_f$ takes values in $\C_I$, $IJI=J$ and $[f_{_{I}}, I id_{\Hq}]= 2 G_f JI$.
 The second equality in \eqref{lemprodkfgSplit} readily follows since $\dfrac{\partial g_{_{I}} }{\partial y} =  I \dfrac{\partial g_{_{I}} }{\partial x} $ for $g$ being slice regular and $JI=-JI$.
\end{proof}


Accordingly, we can prove the following
\begin{theorem}\label{thmProdSpR} If $f,g\in \mathcal{SR}$ such that $f_{_{I}} \in\C_I$ for every $I\in\Sq$, then
	$f\cdot g\in \mathcal{SR}$.
	More generally, $f \cdot g$ is S--polyregular if and only if there exists a positive integer $k_0$ such that for every $I\in\Sq$ we have
	$$  \mathcal{Z}(G_f) \cup \mathcal{Z}\left( \dfrac{\partial^{k_0}g_{_{I}} }{\partial y^{k_0}}\right) = \C_I.$$
	Here $\mathcal{Z}(h)$ denotes the zeros set of $h$.
\end{theorem}

\begin{proof}
	The condition $f\in \mathcal{SR}$ such that $f_{_{I}}\in \C_I$ for every $I\in\Sq$ is clearly equivalent to that $G_f=0$. Hence, by Lemma \ref{lemActionk}, we get $\overline{\partial_{I}}^{k_0} (f\cdot g)_{_{I}}=0$ for arbitrary $I$ and therefore $f\cdot g \in \mathcal{SR}$.
	This completes our check of the first assertion. The second one is immediate from Lemma \ref{lemActionk} and traduces the fact that $G_f J \dfrac{\partial^{k_0} g_{_{I}} }{\partial y^{k_0}}=0$ on $\Hq$ for every $I,J\in\Sq$ such that $J\perp I$.
\end{proof}

\begin{remark}
	An interesting particular case of sufficient condition ensuring the polyregularity of $f\cdot g$ is that $g\in \mathcal{SR}$ and  $\dfrac{\partial^{k_0}g_{_{I}} }{\partial y^{k_0}}=0$  on $\C_I$ for every $I\in\Sq$.
\end{remark}

	The first assertion of Theorem \ref{thmProdSpR} can be reproved using the splitting lemma for slice regular functions. Indeed, with $G_f=0$, one obtains
	$(f\cdot g)_{_{I}} =   F_fF_g + F_fG_g J$ which shows that $f\cdot g \in \mathcal{SR}$, since $F_fF_g$ and $F_fG_g$ are holomorphic functions on $\C_I$. The general case of arbitrary $f,g \in \mathcal{SR}$ leads to the following result reproving and iproving the first assertion of Theorem \ref{thmpolynomial}.

\begin{theorem}\label{thm:char2}
		Let  $f,g\in \mathcal{SR}$ and $F_f,G_f:\C_{I}\longrightarrow \C_{I}$ be two holomorphic functions on $\C_{I}$; $I\in \Sq$, such that
	$f_{_{I}} =F_f + G_f J$ for given imaginary unit $J\perp I$. Assume in addition that $G_f$ is not identically zero on $\C_I$.	Then, $f\cdot g$ is S--polyregular if and only if $g\in \Hq^R[q]$ is a right quaternionic polynomial function in $q$.
\end{theorem}

\begin{proof}	
	Starting from the splitting lemma for slice regular functions, we get
	$$(f\cdot g)_{_{I}}  
	(F_fF_g - G_f \overline{G_g}) + (F_fG_g + G_f \overline{F_g})J ,$$
	where $F_f,G_f, F_g,G_g:\C_{I}\longrightarrow \C_{I}$ are holomorphic functions on $\C_{I}$; $I\in \Sq$, such that
	$f_{_{I}} =F_f + G_f J$ and $g_{_{I}} =F_g + G_g J$ for given imaginary unit $J\perp I$.
	This implies that  $f\cdot g$ is S--polyregular if and only if
	$F_fF_g - G_f \overline{G_g}$ and $F_fG_g + G_f \overline{F_g}$ are polyanalytic functions on $\C_I$ thanks to Proposition \ref{SplitSp} (Proposition 2.5 in \cite{BenElHam2018}). This is in fact equivalent to
	say that $G_g$ and $F_g$ are polynomials functions in $z\in \C_I$ whenever $G_f$ is not identically zero on $\C_I$.
\end{proof}

\section{{The polyregularity of the $\odot ^L_{sp}$--product of S--polyregular functions}}

The special $\odot ^L_{sp}$--product defined on the S--polyregular functions $f(q,\overline{q})=\sum\limits_{j=0}^{m}\overline{q}^j\varphi_j \in \mathcal{SR}_m$ and 
$ g(q,\overline{q})=\sum\limits_{k=0}^{n}\overline{q}^k\psi_k \in \mathcal{SR}_n$ with
$\varphi_j,\psi_k \in \mathcal{SR}$, by
\begin{equation}
(f\odot ^L_{sp} g) (q,\bq) = \sum\limits_{j,k=0}^{m,n} \bq^{j+k} (\varphi_j \cdot \psi_k)(q)
\end{equation}
is one of the more natural products on S--polyregular functions extending the classical dot product on slice regular functions.
In fact, for $f,g\in \mathcal{SR}$ (i.e., $m=n=0$), we have
$$ f\odot ^L_{sp} g  = f\cdot g.$$
Therefore, the S--polyregularity of $f\odot ^L_{sp} g $ is closely connected to the one of
$\varphi_j \cdot\psi_k $ for varying $j=0,1,\cdots,m$ and $k=0,1,\cdots,n$, discussed in the previous sections.
Thus according to Theorems \ref{thmASreg}, \ref{thmpolynomial} and \ref{thmProdSpR}, we assert the following

\begin{theorem} Let $f\in \mathcal{SR}_m$ and $g\in \mathcal{SR}_n$. Then, the function $f\odot ^L_{sp} g $ is S--polyregular
in the following cases:
\begin{enumerate}
\item the $\psi_k$ are right quaternionic polynomial functions
\item $(\varphi_j)_{_{I}} \in\C_I$ for every $I\in\Sq$ and $j=0,1,\cdots,m$
\item there exists a positive integer $k_0$ such that for every $j=0,1,\cdots,m$ and $k=0,1,\cdots,n$, we have
	$$  G_{\varphi_j}  \cdot \dfrac{\partial^{k_0}(\psi_k)_{_{I}} }{\partial y^{k_0}} = 0$$
for every $I\in\Sq$.
\end{enumerate}
\end{theorem}

In conclusion, $\odot ^L_{sp}$ maps $\mathcal{SR}_m \times \mathcal{SR}_n$ into $\bigcup\limits_{s=0}^\infty \mathcal{SR}_{m+n+s}$ for some $s$. Its analog for the polyanalytic functions  is clearly a trivial fact and reads simply $\mathcal{A}_m \cdot \mathcal{A}_n = \mathcal{A}_{m+n}$.

\section{Concluding remark}

In the previous sections, we have studied the polyregularity of the dot product of slice regular functions $f,g\in \mathcal{SR}$. We have provide sufficient condition  as well as a concrete characterization to such product be S--polyregulars. Thus, the polyregularity of $f\cdot g$ is completely determined by $g$, while $f$ has a direct impact in determining the exact level of polyregularity of $f\cdot g$.  Subsequently, we have linearized such products as polynomials in $\bq$  with coefficients are regular functions in $\mathcal{SR}$.
In particular, the function $f\cdot q$ is biregular even $f$ belongs to $\mathcal{SR}$ and therefore $f \cdot q = \bq \varphi_1 + \varphi_0$ for some $\varphi_\ell\in \mathcal{SR}$; $\ell=0,1$.
The inverse 
 problem can also be considered. This means that for given biregular function $\bq \varphi+\psi$, one looks for the existence of $f,g\in \mathcal{SR}$ that solve the functional equation $\bq \varphi+\psi = f \cdot g$.
 Another interesting problem to be considered is motivated by Proposition \ref{HermiteOp} (Proposition 2.4 in \cite{BenElHam2018}) ensuring the existence of $F\in \mathcal{SR}$ such that
 $$f  =(-1)^n e^{|q|^2} \partial^n_I(e^{-|q|^2}F)$$
 for given $f\in \mathcal{SR}_n$. The case of the biregular function $f\cdot q$ with $f\in \mathcal{SR}$ leading to
 $$f \cdot q  = -  e^{|q|^2} \partial_I(e^{-|q|^2}F)$$
 is of particular interest.
 Such type of problems are the subject of a forthcoming paper.



\end{document}